\documentclass[a4paper,12pt]{amsart}

\usepackage[utf8]{inputenc}

\usepackage{geometry}
\usepackage{amsmath,amsfonts,amsthm, amssymb,mathtools}
\usepackage{hyperref, color}

\usepackage{soul}

\newcommand{\rdd}{\mathbb{R}^{2d}}
\newcommand{\rd}{\mathbb{R}^{d}}
\newcommand{\cA}{\mathcal{A}}
\newcommand{\cF}{\mathcal{F}}

\newcommand{\R}{\mathbb{R}}

\newcommand{\bN}{\mathbb{N}}

\newcommand{\cS}{\mathcal{S}}

\let\originalleft\left
\let\originalright\right
\renewcommand{\left}{\mathopen{}\mathclose\bgroup\originalleft}
\renewcommand{\right}{\aftergroup\egroup\originalright}

\newtheorem{theorem}{Theorem}[section]
\newtheorem{corollary}[theorem]{Corollary}
\newtheorem{definition}[theorem]{Definition}

\newtheorem{proposition}[theorem]{Proposition}
\newtheorem{remark}[theorem]{Remark}

\numberwithin{equation}{section}

\title[Dynamical restriction for Schr\"odinger equations]{Dynamical restriction for Schr\"odinger equations}

\author[Fabio Nicola]{Fabio Nicola}
\address{Dipartimento di Scienze Matematiche, Politecnico di Torino, Corso Duca degli Abruzzi 24, 10129 Torino, Italy}
\email{fabio.nicola@polito.it}

\date{}

\begin{document} 
\begin{abstract}
We prove a dynamical restriction principle, asserting that every restriction estimate satisfied by the Fourier transform in $\R^d$ is also valid for the propagator of certain Schr\"odinger equations. We consider smooth Hamiltonians with an at most quadratic growth, and also a class of nonsmooth Hamiltonians, encompassing potentials that are Fourier transforms of complex (finite) Borel measures. Roughly speaking, if the initial datum belongs to $L^p(\R^d)$, for $p$ in a suitable range of exponents, the solution $u(t,\cdot)$ (for each fixed $t$, with the exception of certain particular values) can be meaningfully restricted to compact curved submanifolds of $\R^d$. The underlying property responsible for this phenomenon is the boundedness of the propagator $L^p\to(\mathcal{F}L^p)_{\rm loc}$, with $1\leq p\leq2$, which is derived from almost diagonalization and dispersive estimates in function spaces defined in terms of wave packet decompositions in phase space.
\end{abstract}

\subjclass[2010]{35Q41, 42B10, 42B35, 35S05, 35S10}
\keywords{Restriction estimates, Fourier transform, Schr\"odinger equation, Wiener amalgam spaces, Gabor wave packets}
\maketitle

\section{Introduction and discussion of the main results}

The regularizing effect of the Schrödinger propagator has been object of extensive research and underlies notable \textit{space-time} estimates, like local smoothing and Strichartz estimates (see, e.g., \cite{goldberg_schlag,tao_book}), along with \textit{fixed-time} estimates in weighted Sobolev spaces (see, e.g., \cite{doi}). The existing literature is really vast, and we do not attempt to recount the principal contributions here.

This note explores an additional instance of this regularizing effect, which, to our knowledge, has not been explicitly studied in the literature before.

Let $\mathcal{F}$ denote the Fourier transform in $\R^d$ and let $\nu$ be a positive Radon measure in $\R^d$ --- for example, a surface measure of a compact hypersurface of $\R^d$. Let $L^p(\R^d,\nu)$ denote the associated Lebesgue spaces and use $L^p=L^p(\R^d)$ for the spaces corresponding to the Lebesgue measure. 

The idea at the core of this note is highlighted by the following basic observation.
Consider the free Schr\"odinger equation in $\R\times\R^d$,
\[
i\partial_t u=-\frac{1}{2}\Delta u,
\]
and let
\[
U_0(t)f(x)=\frac{1}{(4\pi i t)^{d/2}}\int_{\R^d} e^{\frac{i|x-y|^2}{4t}}f(y)\, dy
\]
be the corresponding propagator, that is, $u(t,x)=U_0(t)f(x)$ is the solution satisfying $u(0,x)=f(x)$. By expanding the exponent of the integral kernel, considering that the pointwise multiplication by $e^{ia|x|^2}$, with $a\in\R$, is bounded in $L^p(\R^d)$ and $L^q(\R^d,\nu)$, we see that the following facts are equivalent for every $1\leq p,q\leq\infty$, $t\not=0$:
\begin{itemize}
    \item[(a)] $\mathcal{F}:L^p(\R^d)\to L^q(\R^d,\nu)$ continuously; 
    \item[(b)] $U_0(t):L^p(\R^d)\to L^q(\R^d,\nu)$ continuously. 
\end{itemize}

The investigation of conditions on $p, q, \nu$ that ensure $(a)$ is valid represents one of the main research areas in modern harmonic analysis --- the so-called \textit{restriction problem}. This issue has numerous connections with other problems and the literature on it is extensive. We refer the reader to the works \cite{tao,wolff} and the books \cite{demeter_book,stein_1993} for an in-depth exploration of this captivating subject; see also \cite{luef2024fourierrestrictionschattenclass} for a new, noncommutative perspective. In essence, the core philosophy can be summarized by stating that, even though the Fourier transform $\cF f$ of a function $f \in L^p(\rd)$, with $p > 1$, is typically not continuous, if $p$ belongs to a certain range $[1,p_0)$, $\cF f$ can still be meaningfully restricted to a \textit{curved} $k$-dimensional submanifold of $\rd$ --- that is, $\cF$ extends to a bounded operator from $L^p(\rd)$ to some Lebesgue space $L^q(\rd, \nu)$, where $\nu$ is the $k$-dimensional Hausdorff measure restricted to that submanifold. Hence, by the above observation, we see that \textit{the same happens for the solution $u(t,\cdot)=U_0(t)f$ of the free Schr\"odinger equation, with the initial datum $f\in L^p(\R^d)$, for the same range of exponents}. This exemplifies the widely recognized principle that the Schr\"odinger propagator frequently behaves like the Fourier transform (see, e.g., \cite[page 519]{wahlberg2020}).

One may wonder if a similar phenomenon persists when a potential is present. In this note, we affirmatively address this query for two classes of Hamiltonians, that is, smooth Hamiltonians with at most quadratic growth and a class of nonsmooth Hamiltonians, which includes potentials that are Fourier transforms of complex (finite) Borel measures. Hence, we provide a \textit{dynamical} version of the restriction problem, in line with recent interest in exploring dynamic versions of classical results in harmonic analysis; refer, for example, to the contributions \cite{cordero,dancona2025,malinnikova}.

The following proposition shows that the underlying property responsible for this phenomenon is an $L^p(\R^d)\to (\cF L^p)_{\rm loc}(\R^d)$ estimate, for $1\leq p\leq 2$.  
\begin{proposition}\label{pro fond}
Let $A:\cS(\R^d)\to \cS'(\R^d)$ be a linear continuous operator such that, for every function $\varphi\in C^\infty_c(\R^d)$ and $1\leq r\leq 2$, there exists a constant $C>0$ such that
\begin{equation}\label{eq 15ma}
\|\cF (\varphi A f)\|_{L^r}\leq C\|f\|_{L^r}\qquad f\in\cS(\R^d).
\end{equation}
Let $1\leq p\leq 2$, $1\leq q\leq\infty$, and let $\nu$ be a positive Radon measure in $\R^d$ with compact support, such that the Fourier transform extends to a bounded operator $L^p(\R^d)\to L^q(\R^d,\nu)$. Then, for the same $p,q,\nu$, there is a constant $C>0$ such that 
\[
\|A f\|_{L^q(\R^d,\nu)}\leq C\|f\|_{L^p(\R^d)}\qquad f\in\cS(\R^d).
\]
\end{proposition}
The proof is immediate: write $\varphi A=\cF^{-1} \cF \varphi A$, with $\varphi=1$ on the support of $\nu$.
This observation motivates the following definition.
\begin{definition}\label{def rest}
    We say that a linear continuous operator $A:\cS(\rd)\to\cS'(\rd)$ has the restriction property if for every $\varphi\in C^\infty_c(\R^d)$ and $r\in [1,2]$, the inequality \eqref{eq 15ma} holds for some $C>0$.
\end{definition}

Hence, we will focus on the problem of proving the restriction property for the propagator of the aforementioned Schr\"oedinger equations.

\subsection{Smooth Hamiltonians} \label{sec smooth}
Consider a real-valued function $a(t,x,\xi)$, $t\in\R$, $x,\xi\in\rd$, satisfying the following conditions. 
\begin{itemize}
\item[(i)] The map $\R\ni t\mapsto a(t,x,\xi)$ is continuous for every fixed $x,\xi\in\rd$;
\item[(ii)] For every $t\in \R$, $a(t,x,\xi)$ is smooth with respect to $x,\xi$ and, for every $T>0$, $\alpha,\beta\in\bN^{d}$ with $|\alpha|+|\beta|\geq 2$ there exist $C_{\alpha,\beta}>0$ such that 
\[
|\partial_x^\alpha\partial_\xi^\beta a(t,x,\xi)|\leq C_{\alpha,\beta}\qquad |\alpha|+|\beta|\geq 2,\quad x,\xi\in\rd,\ |t|\leq T. 
\]
\end{itemize}
Let $\chi_{t,s}$, $s,t\in\R$, be its associated  Hamiltonian flow, i.e. $z(t):=\chi_{t,s}(z_0)$, $z_0\in\rdd$, is the solution of the  equation \[
\dot{z}(t)=J\nabla_z a(t,z(t)),
\]
with 
\[
J=\begin{pmatrix}
0& I\\
-I& 0
\end{pmatrix},
\]
satisfying the initial condition $z(s)=z_0$. 

Consider the Cauchy problem in $\R\times\rd$:
\begin{equation}\label{eq cauchy}
\begin{cases}
    i \partial_t u=a^w(t) u\\
    u(s)=f
\end{cases}
\end{equation}
where $a^w(t)$ stands for the Weyl quantization of $a(t,\cdot)$. This problem is forward and backward globally wellposed in $L^2(\R^d)$ (see \cite[Proposition 7.1]{tataru}) and we will denote by $U(t,s)$ the corresponding propagator, that is,  $u(t)=U(t,s)f$ is the solution to \eqref{eq cauchy}. We write $(x^{t,s},\xi^{t,s})=\chi_{t,s}(y,\eta)$. Let 
\[
C(t,s):=\inf_{y,\eta\in\rdd}\Big|\det\frac{\partial x^{t,s}}{\partial \eta}(y,\eta)\Big|.
\]

\begin{theorem}\label{thm thm2}
For every $T>0$, $\varphi\in\cS(\R^d)$ and $p\in [1,2]$, there is a constant $C_0>0$ such that, for every $t,s\in [-T,T]$ such that $C(t,s)>0$, 
\begin{equation}\label{eq stima sol}
\|\cF(\varphi U(t,s)f)\|_{L^p}\leq C_0 C(t,s)^{-(\frac{2}{p}-1)} \|f\|_{L^p}\qquad f\in\cS(\rd).
\end{equation}
In particular, for such $t,s$', $U(t,s)$ has the restriction property. 
\end{theorem}
For example, we can combine this result with Proposition \ref{pro fond} and the Thomas-Stein restriction theorem for the $(d-1)$-dimensional sphere $S^{d-1}\subset\R^d$, with surface measure $\nu$ (see e.g. \cite{tao,wolff}), that is, the boundedness $\cF: L^{p_0}(\R^d)\to L^2(\R^d,\nu)$, with $p_0=2(d+1)/(d+3)$. Hence we see that, for all $t\in\R$ such that $C(t,0)>0$, the solution $u(t,\cdot)=U(t,0)f$ satisfies
the estimate
\[
\int_{S^{d-1}} |u(t,\cdot)|^2 d\nu\leq C_0 \|f\|_{L^p}^2,\quad f\in\cS(\R^d),\  1\leq p\leq \frac{2(d+1)}{d+3},
\]
with a constant $C_0$ depending on $t$. 
\begin{remark}
\ 
    \begin{itemize}
\item[1)] Estimate \eqref{eq stima sol} does not extend to the ``exceptional times" where $C(t,s)=0$. For example, for $t=s$, $U(s,s)$ is the identity operator, which obviously does not have any regularizing effect. Indeed, the estimate is expected to degenerate when $C(t,s)=0$. The blow-up rate $C(t,s)^{-(\frac{2}{p}-1)}$ is in general not optimal, but cannot be improved with the approach of this note (see below) --- in fact, it is optimal for a stronger estimate, which we will use as an intermediate step (see \eqref{eq stima1 bis}). Considering our primary motivation, see Proposition \ref{pro fond} above, the precise blow-up rate is a secondary concern for us.
\item[2)] For simplicity, we supposed $\varphi\in\cS(\R^d)$, although it is possible to consider a more general multiplier $\varphi$ that is not necessarily a Schwartz function (see Remark 3.2). Observe that the estimate does not hold without any multiplier, as one already sees for the free Schr\"odinger equation, whose propagator $U(t)$ fails to be continuous $L^p(\R^d)\to\cF L^p(\R^d)$ if $p\not=2$ and $t\not=0$.
\item[3)] The above result is, notably,
of a global nature --- it holds ``beyond caustics" (the points where $C(t,s)=0$). 

\item[4)] In contrast to the Strichartz and local smoothing estimates (see, e.g., \cite{tao_book}), the previous result does not \textit{not} involve time averaging.
\end{itemize}
\end{remark}

\subsection{Nonsmooth Hamiltonians}\label{sec nonsmooth} 

We now consider the same problem for certain nonsmooth Hamiltonians. While we defer a more thorough exploration of this topic for later, here we focus on the simple yet nontrivial instance of a time-independent Hamiltonian of the form
\[
a(x,\xi)=a_2(x,\xi)+a_1(x,\xi)+a_0(x,\xi),
\]
where 
\begin{itemize}
    \item[(iii)]
$a_2(x,\xi)$ is a real-valued quadratic form in $\R^d\times\R^d$ (that is, a second degree real-valued homogeneous polynomial);
\item[(iv)] $a_1\in C^\infty(\rdd)$ is a real-valued function satisfying
\[
|\partial^\alpha_x\partial^\beta_\xi a_1(x,\xi)|\leq C_{\alpha,\beta}\qquad |\alpha|+|\beta|
\geq 1,\quad x,\xi\in\rd;
\]
\item[(v)]
\[
a_0\in M^{\infty,1}(\rdd).
\]
\end{itemize} 
The space $M^{\infty,1}(\rdd)$ is the so-called Sj\" ostrand class (after the work \cite{sjostrand} on the Wiener property of pseudodifferential operators; see also \cite{grochenig,lerner_book}). Its definition will be recalled in Section \ref{sec notation}. To get a sense of this space, the reader is notified that the functions in $M^{\infty,1}(\rd)$ are bounded and have local $\cF L^1$ regularity (uniformly with respect to translations). Our interest in this space, as a class of (possibly pseudodifferential) potentials, lies in the fact that if $V:\R^d\to\mathbb{C}$ is the Fourier transform of a complex (finite) Borel measure, then $a_0(x,\xi):=V(x)$ belongs to $M^{\infty,1}(\R^{2d})$ (\cite{nicola_trapasso}). This class of potentials is a popular choice in several contexts, for example, in the literature about the rigorous construction of the Feynman path integral (see, e.g., \cite{albeverio, ito} and also \cite{nicola_book} for a phase-space approach). 

Let $Sp(d,\R)$ be the group of real $2d\times 2d$ symplectic matrices.   Consider the Hamiltonian flow associated with the quadratic Hamiltonian $a_2(x,\xi)$, that is \[
 \R\ni t\mapsto S_t=
\begin{pmatrix} A_t& B_t\\
C_t& D_t
\end{pmatrix}\in Sp(d,\R),
\]
and let $U(t)=e^{-ta^w}$ be the Schr\"odinger propagator associated with the Hamiltonian $a=a_2+a_1+a_0$, with $a_0,a_1,a_2$ as above. 
\begin{theorem}\label{thm thm3}
Consider a Hamiltonian $a(x,\xi)=a_2(x,\xi)+a_1(x,\xi)+a_0(x,\xi)$, with $a_0, a_1, a_2$ as above. For every $T>0$, $\varphi\in\cS(\R^d)$ and $p\in [1,2]$, there exists a constant $C_0>0$ such that, for every $t\in [-T,T]$ such that ${\rm det}(B_t)\not=0$,
\begin{equation}\label{eq stima sol bis}
\|\cF(\varphi U(t)f)\|_{L^p}\leq C_0 ({\rm det}(B_t))^{-(\frac{2}{p}-1)} \|f\|_{L^p}\qquad f\in\cS(\rd).
\end{equation}  
In particular, for such $t$'s, $U(t)$ enjoys the restriction property. 
\end{theorem}

\subsection{Strategy of the proof} Let us briefly describe the strategy of the proof of  Theorems \ref{thm thm2} and \ref{thm thm3}. First, we prove a transference principle (Proposition \ref{pro1.1}), asserting that the restriction property holds for any operator $A$ that satisfies a refined Hausdorff-Young --- alias dispersive --- inequality, namely $A: W^{p',p}(\R^d)\to W^{p,p'}(\R^d)$, $1\leq p\leq2$, where the spaces $W^{p,q}(\R^d)$ are the so-called \textit{Wiener amalgam spaces}. They are defined in terms of wave packet decompositions in phase space and were first introduced by Feichtinger in \cite{feichtinger} (their definition and properties will be recalled in Section \ref{sec notation} below). Hence, we are led to study the continuity $W^{p',p}(\R^d)\to W^{p,p'}(\R^d)$, $1\leq p\leq 2$, of the above propagators (cf. \eqref{eq stima1 bis}). This will be accomplished by exploiting certain almost diagonalization estimates with respect to wave packets, proved in \cite{CNR_JMP,grochenig,tataru}. In summary, we have the following implications, each of which appears to be of independent interest. 
\[
    \textit{Almost diagonalization with respect to Gabor wave packets }
    \]
    \[
    \Downarrow
    \]
    \[
    \textit{Dispersive estimates in Wiener amalgam spaces}
    \]
    \[
    \Downarrow
    \]
    \[
    \textit{Restriction estimates (for Schr\"odinger)}
    \]
    \par\medskip\noindent
We observe that dispersive estimates in Wiener amalgam spaces, in the simpler case of a quadratic Hamiltonian, have already been obtained (using the explicit formula for the solution) in \cite{CN2008,CN_nach,kato2012,kato2019,toft2023,zhao2024}. 

 In Section \ref{sec micro} we will prove  a microlocal version of the above results, for smooth Hamiltonians of the form $a_2(x,\xi)+a_1(x,\xi)$, with $a_2$ and $a_1$ as in Section \ref{sec nonsmooth}, in terms of the so-called \textit{global} (alias Gabor) wave-front set \cite{hormander,trapasso_rodino,rodino_wahlberg,schulz_wahlberg} (see also \cite{oliaro} for an analytic version). In fact, the almost diagonalization estimates mentioned earlier can be seen as a type of propagation of singularities, offering a phase-space-based explanation for the above regularization effect. 

 We restricted our focus to Hamiltonians of the type described in Sections \ref{sec smooth} and \ref{sec nonsmooth} to illustrate the general approach while avoiding heavy technicalities, but the method of this note can similarly be used to address other Hamiltonians, as those considered in \cite{yajima1991}, including smooth magnetic potentials with at most linear growth, as well as other dispersive equations.

\section{Notation and preliminaries}\label{sec notation}
\subsection{Notation} We denote by $C^\infty_c(\R^d)$ the space of smooth functions with compact support in $\R^d$ and by $\cS(\R^d)$ the space of Schwartz functions, along with the space $\cS'(\R^d)$ of temperate distributions. As in the Introduction, $\cF$ denotes the Fourier transform in $\R^d$, normalized as 
\[
\mathcal{F} f(\xi)=\int_{\R^d} e^{-ix\cdot \xi} f(x)\, dx,\qquad \xi\in\R^d.  
\]
We write $\cF L^p(\R^d)$ for the space of temperate distributions whose (inverse) Fourier transform belongs to $L^p(\R^d)$, equipped with the obvious norm. Also, the notation $(\cF L^p)_{\rm loc}(\R^d)$ stands for the space of all $f\in\cS'(\R^d)$ such that $\varphi f\in \cF L^p(\R^d)$, for every $\varphi\in C^\infty_c(\R^d)$, equipped with the obvious family of seminorms.  

\subsection{Modulation and Wiener amalgam spaces}
For $x,\xi\in\R^d$, $g\in L^2(\R^d)$, define the phase-space shift $\pi(x,\xi)$ by
\begin{equation}\label{eq shift}
\pi(x,\xi) g(y)=e^{i y\cdot \xi}g(y-x)\qquad y\in\rd.
\end{equation}
If $g\in \cS(\rd)\setminus\{0\}$, the functions $\pi(x,\xi) g$ are sometimes called Gabor wave packets, or Gabor atoms, or even coherent states \cite{grochenig_book,lieb_loss_book}.

Given $g\in\cS(\rd)\setminus\{0\}$, we define the short-time Fourier transform of $f\in \cS'(\rd)$ by
\[
V_gf(z)=\langle f,\pi(z)g\rangle \qquad z\in\rdd.
\]
One can verify that, if $\|g\|_{L^2}=(2\pi)^{-d/2}$, $V_g: L^2(\rd)\to L^2(\rdd)$ is a (not onto) isometry,  i.e., $V_g^\ast V_g= I$.

Then, the modulation spaces $M^{p,q}(\rd)$, $1\leq p,q\leq\infty$ (cf. \cite{benyi_book,feichtinger,grochenig_book}), are defined as the set of temperate distributions $f\in \cS'(\rd)$ such that
\[
\|f\|_{M^{p,q}}:=\Big(\int_{\rd} \|V_g f(\cdot,\xi)\|^q_{L^p}\, d\xi\Big)^{1/q}<\infty
\]
(with obvious changes if $q=\infty$). Different $g$'s give rise to equivalent norms. 

Similarly, we consider their Fourier images $W^{p,q}(\rd)=\cF M^{p,q}(\rd)$, $1\leq p,q\leq\infty$, also called Wiener amalgam spaces and often denoted by $W(\cF L^p,L^q)$ in the literature, which can be equivalently defined by requiring that 
\[
\|f\|_{W^{p,q}}:=\Big(\int_{\rd} \|V_g f(x,\cdot)\|^q_{L^p}\, dx\Big)^{1/q}<\infty
\]
(with obvious changes if $q=\infty$). 

We observe that for $p=q$, $M^{p,p}(\rd) = W^{p,p}(\rd)$ and for $p=q=2$, $M^{2,2}(\rd)=W^{2,2}(\rd)=L^2(\rd)$. Moreover, the reader should keep in mind that both $M^{q,p}(\rd)$ and $W^{p,q}(\rd)$ are constituted by (generalized) functions with local $\cF L^p$ regularity and $L^q$ decay at infinity. This makes the following properties intuitively clear (see the original papers \cite{feichtinger, toft} or the books \cite{benyi_book,cordero_book,nicola_book} for the proofs).

\textit{Embeddings.} If $1\leq p_1,p_2,q_1,q_2\leq \infty$, $W^{p_1,q_1}(\rd)\subset W^{p_2,q_2}(\rd)$ if and only if $p_1\leq p_2$ and $q_1\leq q_2$. Moreover if $1\leq p\leq 2$, $L^p(\rd)\subset W^{p',p}(\rd)$ and $W^{p,p}(\rd)\subset L^p(\rd)$. All these inclusions are, in fact, continuous embeddings.  

\textit{Convolution.} If $1\leq p_1,p_2,q_1,q_2,p,q\leq \infty$, and 
\[
\frac{1}{p_1}+\frac{1}{p_2}=\frac{1}{p},\qquad 
\frac{1}{q_1}+\frac{1}{q_2}=\frac{1}{q}+1,
\]
we have 
\[
\|f\ast g\|_{W^{p,q}}\leq C \|f\|_{W^{p_1,q_1}} \|g\|_{W^{p_2,q_2}}.
\]

\textit{Pointwise multiplication.}
If $1\leq p_1,p_2,q_1,q_2,p,q\leq \infty$, and \[
\frac{1}{p_1}+\frac{1}{p_2}=\frac{1}{p}+1,\qquad \frac{1}{q_1}+\frac{1}{q_2}=\frac{1}{q},
\]
we have 
\[
\|f g\|_{W^{p,q}}\leq C \|f\|_{W^{p_1,q_1}} \|g\|_{W^{p_2,q_2}}.
\]

\subsection{Pseudodifferential operators}\label{sec pseudo}
We define the Weyl quantization of a temperate distribution $a(x,\xi)$ in $\R^{2d}$ formally as 
\[
a^w f(x)=(2\pi)^{-d}\int_{\rdd} e^{ i  (x-y)\cdot \xi} a\Big(\frac{x+y}{2},\xi\Big) f(y)\, dy d\xi \qquad f\in\cS(\R^d).
\]
For our purposes, the relevant classes of symbols will be the H\"ormander class $S^0_{0,0}(\rdd)$, constituted by the smooth functions $a\in C^\infty(\rdd)$ such that $\partial^\alpha a\in L^\infty(\rd)$ for every multi-index $\alpha \in\bN^{2d}$, and the class $M^{\infty,1}(\rdd)$ --- the so-called Sj\"ostrand class, which represents the natural symbol class for a noncommutative version of the celebrated Wiener theorem \cite{sjostrand}. 

According to our previous discussion, the functions in $M^{\infty,1}(\rdd)$ are bounded and have a local $\cF L^1$ regularity. Moreover, we have $S^0_{0,0}(\rdd)\subset M^{\infty,1}(\rdd)$.

\textit{Almost diagonalization.} The issue of the almost diagonalization of pseudodiffential and Fourier integral operators has a long tradition that dates back at least to the pioneering work  \cite{fefferman,grochenig,meyer,sjostrand,tataru}. We do not attempt to list the huge number of subsequent contributions, but we address the interested reader to the books \cite{cordero_book,muscalu_book,nicola_book}, the survey \cite{trapasso_ad} and the references therein.

For future reference, it is important to recall that the above symbol classes can be characterized in terms of almost diagonalization estimates for the corresponding Weyl operators. 
\begin{itemize}
    \item[(a)] 
If $a\in \cS'(\rdd)$ then $a\in M^{\infty,1}(\rdd)$ if and only if $a^w$ satisfies the estimate 
\begin{equation}\label{eq almostdiag zero}
    |\langle a^w\pi(z) g,\pi(w)g\rangle |
    \leq H(w-z)\qquad z,w\in\rdd
\end{equation}
for some $H\in L^1(\rdd)$ and some (and therefore, every) $g\in\cS(\rd)\setminus\{0\}$ \cite{grochenig}.
\item[(b)]
If $a\in \cS'(\rdd)$ then $a\in S^0_{0,0}(\rdd)$ if and only if 
\begin{equation}
    |\langle a^w\pi(z) g,\pi(w)g\rangle |
    \leq C_N(1+|w-z|)^{-N}\qquad z,w\in\rdd
\end{equation}
for every $N\in\bN$ and some (and therefore, every) $g\in\cS(\rd)\setminus\{0\}$ \cite{tataru} (that is, if \eqref{eq almostdiag zero} holds with $H(z)=C_N(1+|z|)^{-N}$ for every $N\in\bN$).
\end{itemize}
These characterizations imply that such symbol classes are closed with respect to the Weyl product (see below) and also imply at once corresponding continuity properties. 

\textit{Continuity.} 
If $a\in M^{\infty,1}(\rdd)$ then $a^w: M^{p,q}(\rd)\to M^{p,q}(\rd)$ continuously, for $1\leq p,q\leq\infty$ \cite{grochenig} (in particular $a^w: L^2(\rd)\to L^2(\rd)$). This is also an immediate consequence of the above characterization (a) and the formula \eqref{eq lift} below, using Young's inequality for mixed norm spaces \[
L^1(\rdd)\ast L^q(\rd; L^p(\rd))\subset L^q(\rd; L^p(\rd)).
\]
The same argument shows that if $a\in M^{\infty,1}(\rdd)$ then $a^w: W^{p,q}(\rd)\to W^{p,q}(\rd)$ continuously, for $1\leq p,q\leq\infty$. 

We will also make use of the following Banach algebra property \cite{grochenig,sjostrand}. 
\begin{itemize}
    \item[(c)] The set of Weyl operators $a^w$, with Weyl symbol $a\in M^{\infty,1}(\R^{2d})$, endowed with the norm $\|a^w\|:=\|a\|_{M^{\infty,1}}$, is a Banach algebra of bounded operators in $L^2(\R^d)$. An equivalent norm is given by 
    \begin{equation}\label{eq equiv norm}
|||a^w|||:=\int_{\R^{2d}}\sup_{w\in\R^{2d}}|\langle a^w \pi(z+w)g,\pi(w)g\rangle|\, dz,
    \end{equation}
    for any $g\in\cS(\R^d)\setminus\{0\}$ (cf. item (a) above). That is, for $a\in\cS'(\R^{2d})$, 
    \[
    C^{-1} \|a\|_{M^{\infty,1}}\leq |||a^w|||\leq C \|a\|_{M^{\infty,1}}
    \]
    for some constant $C>0$. 

\end{itemize}

Incidentally, we observe that lower bounds for pseudodifferential operators for these symbol classes have also been studied; for example, the classical Fefferman-Phong inequality is known to hold for a nonnegative symbol with fourth derivatives in $M^{\infty,1}(\rdd)$ \cite{lerner_book}.

\section{A transference principle}
In this section we show that any operator satisfying a refined form of the Hausdorff-Young inequality enjoys the restriction property (see Definition \ref{def rest}).

The classical Hausdorff-Young inequality asserts that the Fourier transform extends to a bounded operator $L^p(\rd)\to L^{p'}(\rd)$, for $1\leq p\leq 2$, $\frac{1}{p}+\frac{1}{p'}=1$. In fact, Feichtinger proved (in the more general context of locally compact Abelian groups) that, for $1\leq p\leq q\leq\infty$, $\cF:W^{q,p}(\rd)\to W^{p,q}(\rd)$ continuously \cite{feichtinger_inter}. Since $L^p(\rd)\hookrightarrow W^{p',p}(\rd)$ and $W^{p,p'}(\rd)\hookrightarrow L^{p'}(\rd)$ for $1\leq p\leq 2$, this result represents an extension and an improvement of the Hausdorff-Young inequality. 

We can state our first result.
\begin{proposition}[Transference principle]\label{pro1.1}
Let $1\leq p\leq 2$ and $\varphi\in \cS(\rd)$. There is a constant $C>0$ such that, for every bounded operator $A:W^{p',p}(\rd)\to W^{p,p'}(\rd)$, we have 
\begin{equation}\label{eq cont}
\|\mathcal{F}(\varphi Af)\|_{L^p}\leq C \|A\|_{W^{p',p}\to W^{p,p'}} \|f\|_{L^p}\qquad f\in\cS(\rd).
\end{equation}
As a consequence, $A$ has the restriction property (see Definition \ref{def rest}).
\end{proposition}
\begin{proof}[Proof of Proposition \ref{pro1.1}] 
Since $1\leq p\leq 2$, we have $W^{p,p}(\rd)\hookrightarrow L^p(\rd)$ and moreover the Fourier transform is bounded in $W^{p,p}(\rd)$. Hence 
\[
\|\cF(\varphi Af)\|_{L^p}\leq C\|\varphi Af\|_{W^{p,p}}.
\]
Observe that $\varphi\in\cS(\rd)\subset W^{1,q}(\rd)$ for every $q\in[1,+\infty]$. Choosing $q$ so that $\frac{1}{q}=\frac{1}{p}-\frac{1}{p'}$, by the pointwise multiplication property of Wiener amalgam spaces (see Section \ref{sec notation}) we have
\begin{align*}
\|\varphi Af\|_{W^{p,p}}&
\leq C' \|\varphi\|_{W^{1,q}}\|Af\|_{W^{p,p'}}\\
&\leq C'' \|A\|_{W^{p',p}\to W^{p,p'}}\|f\|_{W^{p',p}}.
\end{align*}
Since $L^p(\rd)\hookrightarrow W^{p',p}(\rd)$ for $1\leq p\leq 2$, this concludes the proof of \eqref{eq cont}. 
\end{proof}
\begin{remark}\label{rem 3.2}
It follows from the above proof that \eqref{eq cont} holds true  for $\varphi\in W^{1,q}(\rd)$, with $\frac{1}{q}=\frac{1}{p}-\frac{1}{p'}$. 
\end{remark}

\section{Almost diagonalization} 

The proof of Theorems \ref{thm thm2} and \ref{thm thm3} is based on some almost diagonalization estimates from \cite{CNR_JMP,grochenig,tataru}. 
We focus on almost diagonalization with respect to Gabor wave plackets $\pi(z) g$, for $g\in \cS(\R^d)\setminus\{0\}$, $z\in\R^{2d}$, following \cite{CNGR_JMPA,muscalu_book,tataru}. 

Let $\chi:\rdd\to\rdd$ be a smooth ``tame" symplectomorphism, namely:
\begin{itemize}
    \item[(a)] $\chi$ is smooth;
    \item[(b)] We have \begin{equation}\label{eq der chi}
|\partial^\alpha \chi(z)|\leq C_\alpha \qquad z\in\rdd, \ |\alpha|\geq 1; 
\end{equation} 
    \item[(c)] $\chi$ preserves the symplectic structure of $\rdd$, i.e. with $(x,\xi)=\chi(y,\eta)$, $dx\wedge d\xi= dy\wedge d\eta$. 
\end{itemize}
 The Hamiltonian flow $\chi_{t,s}$, for every fixed $t,s\in\R$, considered in Section \ref{sec smooth}, is a typical example of a tame symplectomorphism.

We consider a linear continuous operator 
$A:\cS(\rd)\to\cS'(\rd)$ satisfying the estimate
\begin{equation}\label{eq almostdiag}
    |\langle A\pi(z) g,\pi(w)g\rangle |
    \leq H(w-\chi(z))\qquad z,w\in\rdd,
\end{equation}
for some $H\in L^1(\rdd)$, $g\in \cS(\rd)\setminus\{0\}$. 
\begin{theorem}\label{thm thm1}
Suppose that $A:\cS(\rd)\to\cS'(\rd)$ satisfies \eqref{eq almostdiag} for some $H\in L^1(\rd)$, $g\in \cS(\rd)\setminus\{0\}$, and $\chi$ as above.
Let $(x,\xi)=\chi(y,\eta)$ and suppose 
\[
B:=\inf_{y,\eta\in\rdd}\big|\det\frac{\partial x}{\partial \eta}(y,\eta)\big|>0.
\]
For $1\leq p\leq q\leq\infty$, there exists a constant $C>0$ depending only on $g$ (and the choice of the window in the definition of the spaces $W^{p,q}$) such that 
\[
\label{eq stima1}
\|Af\|_{W^{p,q}}\leq  C B^{-(\frac{1}{p}-\frac{1}{q})}\|H\|_{L^1}\|f\|_{W^{q,p}}\qquad f\in\cS(\R^d).
\]

\end{theorem}

\begin{proof}
We can suppose that the function $g\in\cS(\rdd)\setminus\{0\}$ in \eqref{eq almostdiag} is normalized so that $\|g\|_{L^2}=(2\pi)^{-d/2}$. Moreover, we will use $g$ as a window in the definition of the spaces $W^{p,q}(\R^d)$ (see Section \ref{sec notation}). 

 We use the inversion formula $f=V_g^\ast V_g f$:
\[
f(y)=\int_{\rdd} V_g f(z)\pi(z) g (y)\,dz\qquad y\in\R^d.
\]
Applying the operator $A$ we have 
\begin{equation}\label{eq lift}
   V_g(Af)(w)=\int_{\rdd}\langle A\pi(z)g,\pi(w)g\rangle V_gf(z)\, dz. 
   \end{equation}
   Using \eqref{eq almostdiag} we see that 
\[
|V_g(Af)(w)|\leq \int_{\rdd} H(w-\chi(z))|V_g f(z)|\, dz. 
\]
Since $\chi:\rdd\to\rdd$ is a symplectomorphism, its derivative has determinant $1$, so that
\[
|V_g(Af)(w)|\leq \int_{\rdd} H(w-z)|V_g f(\chi^{-1}(z))|\, dz. 
\]
By Young's inequality for mixed-norm Lebesgue spaces we deduce that 
\[
\|V_g(Af)\|_{L^{q}_x (L^p_\xi)}\leq \|H\|_{L^1} \|(V_g f)\circ \chi^{-1}\|_{L^{q}_x (L^{p}_\xi)}. 
\]
Hence it is sufficient to prove that 
\[
\|F\circ \chi^{-1}\|_{L^{q}_x (L^p_\xi)}\leq B^{-(\frac{1}{p}-\frac{1}{q})}\|F\|_{L^{p}_y (L^{q}_\eta)} 
\]
for (continuous) functions $F(y,\eta)$ on $\rdd$. This is clear if $p=q$. By interpolation it is enough to prove the case $1\leq p<\infty$, $q=\infty$, namely
\begin{equation}\label{eq da pro}
\sup_{x\in\rd}\int_{\rd}|F(y(x,\xi),\eta(x,\xi))|^p\, d\xi\leq B^{-1}\int_{\rd}
\, \sup_{\eta\in\rd}|F(y,\eta)|^p\,dy.
\end{equation}
Now, we have 
\[
\int_{\rd}|F(y(x,\xi),\eta(x,\xi))|^p\, d\xi\leq \int_{\rd}\sup_{\eta\in\rd}|F(y(x,\xi),\eta)|^p\, d\xi.
\]
Moreover, by Hadamard's global inverse function theorem (see, e.g., \cite[Theorem 6.2.4]{krantz}), we can perform the chance of variable $y=y(x,\xi)$ in the right-hand side, regarding $x$ as a parameter. Indeed, since the inverse of a symplectic matrix $M\in Sp(d,\R)$ is given by $J^{-1} M^t J$, we see that 
\[
\frac{\partial y}{\partial \xi}(x,\xi)= - \Big(\frac{\partial x}{\partial \eta}\Big)^t(y,\eta).
\]
Hence
\[
\big|{\rm det}\frac{\partial y}{\partial \xi}(x,\xi)\big|\geq B>0,
\]
and we deduce
\[
\int_{\rd}\sup_{\eta\in\rd}|F(y(x,\xi),\eta)|^p\, d\xi\leq B^{-1} \int_{\rd}\sup_{\eta\in\rd}|F(y,\eta)|^p\, dy,
\]
which concludes the proof of \eqref{eq da pro}. 
\end{proof}

\section{Alomost diagonalizations of Scr\"odinger propagators} 
This section is devoted to almost diagonalization results for the Schr\"odinger equations considered in Sections \ref{sec smooth} and \ref{sec nonsmooth}.

We recall the following result from \cite{tataru} (where a Gaussian window was considered, but the same result holds for every $g\in\cS(\R^d)\setminus\{0\})$ as a consequence of the results in \cite{CNGR_JMPA}). 
\begin{theorem}\label{tataru} 
    Let $a(x,\xi)$ be a Hamiltonian that satisfies assumptions (i) and (ii) of Section \ref{sec smooth}, let $\chi_{t,s}$ be the corresponding Hamiltonian flow and let $U(t,s)$ be the corresponding Schr\"odinger propagator. Let $g\in\cS(\R^d)\setminus\{0\}$. Then for every $T>0$ and $N\in\bN$ there exists $C_N>0$ such that \[
|\langle U(t,s)\pi(z) g,\pi(w)g\rangle |
    \leq C_N (1+|w-\chi_{t,s}(z)|)^{-N}\qquad z,w\in\rdd,\ t,s\in[-T,T].
\]
\end{theorem}
We now focus on the class of nonsmooth Hamiltonians considered in Section \ref{sec nonsmooth}. First, we need the following result, which follows from the proof of \cite[Theorem 3.4]{CNR_JMP}.
\begin{proposition}\label{pro composiz}
    Let $g\in\cS(\R^d)$, with $\|g\|_{L^2}=(2\pi)^{-d/2}$. Let $A_j:\cS'(\R^d)\to\cS'(\R^d)$, $j=1,2$, be linear continuous operators satisfying the estimates 
    \[
    |\langle A_j \pi(z)g,\pi(w)g\rangle|\leq H_j (w-\cA_j z),\qquad z,w\in\R^{2d}
    \]
    for some $H_j\in L^1(\R^{2d})$, $\cA_j\in Sp(d,\R)$, $j=1,2$. 
    
    Then 
    \[
    |\langle A_1A_2 \pi(z)g,\pi(w)g\rangle|\leq H (w-\cA_1\cA_2 z),\qquad z,w\in\R^{2d},
    \]
  where $H\in L^1(\R^{2d})$ is given\footnote{
Here the matrices $\cA_j$ are identified with the corresponding linear maps, $\circ$ stands for the composition of functions, and $\ast$ for the convolution.
  }
  by $H=((H_1\circ\cA_1)\ast H_2)\circ \cA_1^{-1}$. In particular, $\|H\|_{L^1}\leq \|H_1\|_{L^1}\|H_2\|_{L^1}$. 
\end{proposition}

\begin{theorem}\label{pro 8mag} 
Consider a Hamiltonian $a(x,\xi)=a_2(x,\xi)+a_1(x,\xi)+a_0(x,\xi)$, with $a_0,a_1,a_2$ satisfying assumptions (iii), (iv), and (v) of Section \ref{sec nonsmooth}. Let $\R\ni t\mapsto S_t$ be the Hamiltonian flow of $a_2$ and let $U(t)=e^{-it a^w}$ be the corresponding Schr\"odinger propagator, as in Section \ref{sec nonsmooth}. Let $g\in\cS(\R^d)\setminus\{0\}$. Then
\begin{equation}\label{eq al}
|\langle U(t)\pi(z) g,\pi(w)g\rangle |
    \leq H_t(w-S_t z)\qquad z,w\in\rdd,\quad t\in\R
\end{equation}
for some time-dependent function $H_t\in L^1(\R^{2d})$, with $\|H_t\|_{L^1}$ locally bounded as a function of $t\in\R$. 
\end{theorem}
\begin{proof}
This result was already proved in \cite[Theorem 4.1]{CNR_JMP} in the case $a_1=0$, hence $a=a_2+a_0$, regarding the nonsmooth term $a_0\in M^{\infty,1}(\R^{2d})$ as a perturbation. Later, that proof was slightly simplified in \cite[Theorem 6.4.1]{nicola_book}. There we worked in the so-called interaction picture and therefore we studied the evolution of the operator
\[
e^{it a_2^w} a_0^w e^{-ita_2^w},
\]
a key point being that this is still a pseudodifferential operator with symbol in $M^{\infty,1}(\R^{2d})$, as a consequence of the symplectic covariance of the Weyl quantization. This is no longer a priori guaranteed if $a_2$ is replaced by $a_2+a_1$ in the above formula (in fact, this is still the case, as a consequence of the result we are proving, but we cannot use this fact yet). Consequently, we will use an approach similar to that in \cite{CNR_JMP,nicola_book}, but several modifications will be necessary at certain crucial points.

Let 
\[
U_1(t)=e^{-it(a_2+a_1)^w}
\]
be the propagator associated with the Hamiltonian $a_2+a_1$, and let $\sigma_t\in\cS'(\R^{2d})$, $t\in\R$, be defined by 
\[
\sigma_t^w:=U_1(t) a_0^w U_1(-t).
\]
Since $a_0\in M^{\infty,1}(\R^{2d})$, $\sigma_0^w$ is a bounded operator in $L^2(\R^d)$ and $\sigma_t^w$ is a strongly continuous family of bounded operators in $L^2(\R^d)$. 

We will prove, in order, the following facts. 
\begin{itemize}
        \item[(a)] For every $T>0$ and $N\in\bN$ there exists $C_N>0$ such that 
    \begin{equation}\label{eq 5ma}
|\langle U_1(t)\pi(z) g,\pi(w)g\rangle |
    \leq C_N (1+|w-S_t z|)^{-N}\qquad z,w\in\rdd,\ t\in[-T,T].
      \end{equation}
      \item[(b)]  We have 
      \[
|\langle \sigma^w_t \pi(z) g,\pi(w)g\rangle |
    \leq H_t (w-z)\qquad z,w\in\rdd,\ t\in\R,
      \]
      for some function $H_t\in L^1(\R^{2d})$, with $\|H_t\|_{L^1}$ locally bounded, as a function of $t$. Equivalently, for every $t\in\R$, $\sigma_t\in M^{\infty,1}(\R^{2d})$, and for every $T>0$ there exists $C>0$ such that 
      \begin{equation}\label{eq 9ma}
\|\sigma_t\|_{M^{\infty,1}}\leq C,\qquad t\in[-T,T]. 
    \end{equation}
    
      \item[(c)] For every $t\in\R$, $k\geq1$, the iterated integral 
      \[
b_{t,k}^w:=\int_{0}^t\int_0^{t_1}\ldots\int_0^{t_{k-1}} \sigma^w_{t_1}\sigma^w_{t_2}\ldots \sigma^w_{t_k}\, dt_k\, dt_{k-1}\ldots dt_1,
      \]
which converges in the strong topology of bounded operators in $L^2(\R^d)$, defines an operator with Weyl symbol $b_{t,k}\in M^{\infty,1}(\R^{2d})$, and for every $T>0$ there exists $C>0$ such that
\begin{equation}\label{eq 9ma bis}
\|b_{t,k}\|_{M^{\infty,1}}\leq C^{k}/k!,\qquad t\in [-T,T].
\end{equation}
\end{itemize}

Let us now prove (a). 
It follows from the almost diagonalization result in Theorem \ref{tataru}, applied to the Hamiltonian $a_2(x,\xi)+a_1(x,\xi)$ that, for every $T>0$ and $N\in\bN$ there exists $C_N>0$ such that 
\begin{equation}\label{eq tataru}
|\langle U_1(t)\pi(z) g,\pi(w) g\rangle|\leq C_N(1+|w-\chi_t(z)|)^{-N} \qquad z,w\in\rdd,\ t\in[-T,T], 
\end{equation}
 where $\chi_t$ is the Hamiltonian flow of the Hamiltonian $a_2(x,\xi)+a_1(x,\xi)$. Hence it is enough to prove that \eqref{eq tataru} holds with $\chi_t(z)$ replaced by $S_t z$, which is the flow associated with $a_2(x,\xi)$ only. To see this, observe that
 \[
 \frac{d}{dt}\chi_t(z)=J\nabla a_2(\chi_t(z))+J\nabla a_1(\chi_t(z)) 
 \]
 and similarly 
 \[
 \frac{d}{dt} S_t z= J\nabla a_2(S_t z).
 \]
 Hence, since $\chi_0(z)=S_0z=z$,
\[
\chi_t(z)-S_t z =\int_0^t \Big(J\nabla a_2(\chi_\tau(z))- J\nabla a_2(S_\tau z)+ J\nabla a_1(S_\tau z)\Big)\, d\tau.
\]
Now, $J\nabla a_2$ is a linear map, whereas $J\nabla a_1$ has entries in $L^\infty(\rd)$ by the assumption on $a_1$ in (iv),  Section \ref{sec nonsmooth}. Hence by  Gronwall's inequality we deduce that 
\[
|\chi_t(z)-S_t z |\leq C|t|e^{C|t|}
\]
for some constant $C>0$ \textit{independent of} $z$. By the triangle inequality we therefore see that we can replace $\chi_t(z)$ by $S_t z$ in \eqref{eq tataru}, which concludes the proof of (a).  

Let us prove (b). 
Since $U_1(t)$ satisfies the estimate \eqref{eq 5ma}, while $a_0^w$ satisfies \eqref{eq almostdiag zero}, it follows from Proposition \ref{pro composiz} that the composition $\sigma_t^w$ enjoys the estimate in (b). The claimed equivalence in terms of the Weyl symbol $\sigma_t$, namely \eqref{eq 9ma}, is a consequence of item (c) in Section \ref{sec pseudo}. 

Now, prove (c). To verify $b_{t,k}\in M^{\infty,1}(\R^d)$, and to estimate its norm, we can use the equivalent norm in \eqref{eq equiv norm}. 
We observe that, since the operators $\sigma_t^w$ and the phase space shifts $\pi(z)$ are strongly continuous in $L^2(\R^d)$, the supremum
\[
\sup_{w\in\R^{2d}}| \langle \sigma^w_{t_1}\sigma^w_{t_2}\ldots \sigma^w_{t_k}\pi(z+w)g,\pi(w)g\rangle|
\]
is a lower semicontinous function of $t_1,\ldots t_k,z$, hence Borel measurable. As a consequence, we can apply Fubini's theorem to exchange the integrals with respect to the $t$'s and with respect to $z$, and we deduce that
\[
|||b_{t,k}|||\leq \big|\int_{0}^t\int_0^{t_1}\ldots\int_0^{t_{k-1}} |||\sigma^w_{t_1}\sigma^w_{t_2}\ldots \sigma^w_{t_k}|||\,\, dt_k\, dt_{k-1}\ldots dt_1\big|
\]
The desired estimate \eqref{eq 9ma bis} then follows from the Banach algebra property and \eqref{eq 9ma} (see item (c) in Section \ref{sec pseudo}). 

Now, regarding $a_0^w$ as a perturbation of $(a_0+a_1)^w$, bounded in $L^2(\R^d)$, by a standard argument from operator theory (cf., e.g., \cite[Theorem 6.4.1]{nicola_book}) we can write $U(t)=e^{-ita^w}$ as
\begin{equation}\label{eq ut}
U(t)=U_1(t) c_t^w,
\end{equation}
where $c_t^w$ is given by the so-called Dyson series 
\begin{align*}
c_t^w&=I+\sum_{k=1}^\infty (-i)^k \int_{0}^t\int_0^{t_1}\ldots\int_0^{t_{k-1}} \sigma^w_{t_1}\sigma^w_{t_2}\ldots \sigma^w_{t_k}\, dt_k\, dt_{k-1}\ldots dt_1\\
&=I+\sum_{k=1}^\infty (-i)^k b_{t,k}^w,
\end{align*}
which converges in the space $\mathcal{B}(L^2(\R^d))$ of bounded operators in $L^2(\R^d)$, endowed with the strong operator topology. However, by \eqref{eq 9ma bis}, this series also converges absolutely in the Banach algebra of pseudodifferential operators with Weyl symbol in $M^{\infty,1}(\R^{2d})$, which is continuously embedded in $\mathcal{B}(L^2(\R^d))$. Hence, we have $c_t\in M^{\infty,1}(\R^{2d})$, and $\|c_t\|_{M^{\infty,1}(\R^{2d})}\leq C(T)$ for $|t|\leq T$. This implies (see again item (c) from Section \ref{sec pseudo}), that 
 \[
|\langle c^w_t \pi(z) g,\pi(w)g\rangle |
    \leq H_t (w-z)\qquad z,w\in\rdd,\ t\in\R,
      \]
      for some function $H_t\in L^1(\R^{2d})$, with $\|H_t\|_{L^1}$ locally bounded, as a function of $t$. Combining this latter estimate with \eqref{eq 5ma}, we deduce again from Proposition \ref{pro composiz} that the operator $U(t)$, as given in \eqref{eq ut}, satisfies the estimate in \eqref{eq al} and this concludes the proof of Theorem \ref{pro 8mag}. 
\end{proof}

\section{Proof of the main results (Theorems \ref{thm thm2} and \ref{thm thm3})}
We are now able to prove our main results. 
\begin{proof}[Proof of Theorem \ref{thm thm2}] As a consequence of the almost diagonalization result in Theorem \ref{tataru} and Theorem \ref{thm thm1} we deduce the following \textit{dispersive estimate} for the propagator $U(t,s)$: for every $1\leq p\leq q\leq\infty$ there exists a constant $C_0>0$ such that  
\begin{equation}\label{eq stima1 bis}
\|U(t,s)f\|_{W^{p,q}}\leq C_0  C(t,s)^{-(\frac{1}{p}-\frac{1}{q})}\|f\|_{W^{q,p}}
\end{equation}
for all $s,t\in[-T,T]$ such that  $C(t,s)>0$. The estimate \eqref{eq stima sol} then follows from Proposition \ref{pro1.1}. 
\end{proof}  
\begin{proof}[Proof of Theorem \ref{thm thm3}]
The proof follows the same pattern as that of Theorem \ref{thm thm2}. That is, now we apply Theorem \ref{pro 8mag}, Theorem \ref{thm thm1} and Proposition \ref{pro1.1}.
\end{proof}

\section{Microlocal restriction estimates}\label{sec micro}
Let $A:\cS'(\rd)\to\cS'(\rd)$ be a continuous linear operator, satisfying the estimate 
\begin{equation}\label{eq almost diag smooth}
|\langle A\pi(z) g,\pi(w) g\rangle|\leq C_N(1+|w-\cA z|)^{-N}\qquad z,w\in\rdd
\end{equation}
for some $\cA\in Sp(d,\R)$, $g\in\cS(\rd)\setminus\{0\}$ and every $N\in \mathbb{N}$.  For such an operator, we will prove a microlocal improvement of \eqref{eq cont}. 

We begin by recalling from \cite{hormander,rodino_wahlberg} the definition of the \textit{global} (alias Gabor) wave front set $WF_G(f)\subset\rdd\setminus\{0\}$ of a temperate distribution $f\in\cS'(\rd)$, as well as its role in the problem of propagation of singularities.  

\textit{Global wave front set.} We say that $z\in \rdd\setminus\{0\}$ does not belong to $WF_G(f)\subset\rdd\setminus\{0\}$ if there exists an open conic neighborhood $\Gamma\subset\rdd\setminus\{0\}$ of $z$ such that $V_g f$ has rapid decay in $\Gamma$ for some $g\in\cS(\rd)\setminus\{0\}$, namely 
\[
\sup_{z\in \Gamma}\,(1+|z|)^N |V_g f(z)|<\infty. 
\]
Equivalently, $z\not\in WF_G(f)$ if the exists a smooth function $\psi(z)$ in $\rdd$, positively homogeneous of degree $0$ for $|z|\geq 1$ (in the sense that coincides for $|z|\geq 1$ with a function positively homogeneous of degree $0$) and such that $\psi(\lambda z)\not=0$ for $\lambda>0$ large enough, such that $\psi^w f\in\cS(\rd)$. 

It turns out that $WF_G(f)$ is a closed conic subset of $\rdd\setminus\{0\}$, and $WF_G(f)=\emptyset$ if and only if $f\in\cS(\rd)$. Hence $WF_G(f)$ encodes both the lack of regularity and decay of $f$. Despite the similarity with the usual wave front set \cite{hormander_book}, there are some striking differences; for example, the global wave front set is invariant under phase-space shifts: 
\[
WF_G(\pi(z)f)=WF_G(f)\qquad z\in\rdd.
\]

\textit{Microlocality.} Given $a\in \cS'(\rdd)$, we define its cone support ${\rm conesupp}(a)\subset\rdd\setminus\{0\}$ as the complementary in $\rdd\setminus\{0\}$ of the set of the $z$'s such that there exists an open conic neighborhood $\Gamma\subset\rdd\setminus\{0\}$ of $z$ such that, for every $N\in\mathbb{N}$ and every multi-index $\alpha$, 
\[
\sup_{z\in \Gamma}\, (1+|z|)^N |\partial^\alpha a(z)|<\infty. 
\] Therefore ${\rm conesupp}(a)$ is a closed conic subset of $\rdd\setminus\{0\}$. For example, if $\varphi(x)$ is a Schwartz function, regarded as a function in $\rdd$ (independent of the covariable $\xi$) we have 
\begin{equation}\label{eq chi}
{\rm conesupp}(\varphi)=\{0\}\times(\rd\setminus\{0\}).
\end{equation}
The reader is notified that the set ${\rm conesupp}(a)$ is called \textit{microsupport} in \cite{SW} and alternative terminology is also used in the literature (see the footnote in \cite[page 95]{SW}).

As expected, if $a\in S^0_{0,0}(\rdd)$ and $f\in \cS'(\rd)$ we have 
\begin{equation}\label{eq micro}
WF_G(a^w f)\subset WF_G(f)\cap {\rm conesupp}(a).
\end{equation}
This was proved in \cite{hormander} for $a$ in Shubin's symbol classes and in \cite{rodino_wahlberg} for $a\in S^0_{0,0}(\rdd)$, except for a slightly different definition of ${\rm conesupp}(a)$. This set is defined there (following \cite{hormander}) as the complementary in $\rdd\setminus\{0\}$ of the set of $z$' such that there exists an open conic neighborhood $\Gamma\subset\rdd\setminus\{0\}$ of $z$ such that
\[ \overline{\Gamma \cap {\rm supp}(a)}\quad \textrm{is compact}. 
\]
With our definition of ${\rm conesupp}(a)$, \eqref{eq micro} is slightly stronger (and perhaps more natural --- for example, we recapture the basic fact that $a\in\cS(\rdd)$ and $f\in\cS'(\rd)$ imply $a^w f\in\cS(\rd)$). However, \eqref{eq micro} follows immediately from the corresponding result in \cite{rodino_wahlberg}, because if $a\in S^0_{0,0}(\rdd)$ has a rapid decay in an open conic neighborhood $\Gamma\subset\rdd\setminus\{0\}$ of $z_0$, for every conic open neighborhood $\Gamma'\subset \subset \Gamma$ there exists $\tilde{a}(z)$ that vanishes for $z$ in $\Gamma'$ and $|z|\geq 1$, such that $a-\tilde{a}\in\cS(\rdd)$, so that $(a-\tilde{a})^w f\in\cS(\rd)$ if $f\in\cS'(\rd)$. 

\textit{Propagation of singularities.} If $A:\cS'(\rd)\to\cS'(\rd)$ is a linear continuous operator satisfying the almost diagonalization estimate \eqref{eq almost diag smooth} for every $N$, for $f\in\cS'(\rd)$ we have
\begin{equation}\label{eq propag}
WF_G(Af)\subset \cA(WF_G(f)),
\end{equation}
where the matrix $\cA$ is identified with the corresponding linear map. Indeed, this is an immediate consequence of the formula \eqref{eq lift}: using \eqref{eq almost diag smooth} we obtain
\[
|V_g (Af)(\cA w)|\leq C_N \int_{\rdd}\langle 1+|w-z|)^{-N}|V_g f(z)|\, dz,
\]
from which it is easy to show that if $V_g f(z)$ has a rapid decay in an open conic neighborhood $\Gamma$, and if $\Gamma'$ is an open conic subset of $\R^{2d}\setminus\{0\}$, with $\Gamma'\subset \subset \Gamma$, then $V_g (Af)(\cA w)$ has a rapid decay for $w\in \Gamma'$ (see \cite{rodino_wahlberg}).

We can then state the promised microlocal improvement of \eqref{eq cont}.
\begin{theorem}\label{thm micro}
    Let $A:\cS'(\rd)\to\cS'(\rd)$ be a continuous linear operator satisfying the estimates \eqref{eq almost diag smooth} for every $N\in\bN$ and some $g\in\cS(\rd)\setminus\{0\}$, where $\cA=\begin{pmatrix}
    A& B\\
    C& D
\end{pmatrix}\in{\rm Sp}(d,\R)$ is a symplectic matrix with $\det B\not=0$. Let $\psi(z)$, with $z\in\R^{2d}$, be a smooth function, positively homogeneous of degree $0$ for $|z|\geq 1$, with $\psi(z)=1$ for $z$ in an open conic neighborhood of  $\cA^{-1}(\{0\}\times(\rd\setminus\{0\}))$ and $|z|\geq 1$.

If $\varphi\in\cS(\rd)$, $1\leq p\leq 2$, there exists $C>0$ and, for every $N\in\bN$, there exists $C_N>0$ such that
\[
\|\cF (\varphi A f)\|_{L^p}\leq C\|\psi^w f\|_{L^p}+C_N\|\langle x\rangle^{-N} f\|_{H^{-N}}\qquad f\in\cS(\R^d),
\]
where $\langle x\rangle:=(1+|x|^2)^{1/2}$ and $\|\cdot\|_{H^{-N}}$ denotes the ($L^2$-based) Sobolev norm of order $-N$ in $\rd$. 
\end{theorem}
\begin{proof}
We write $f=\psi^w f+(1-\psi)^w f$. The term $\cF \varphi A \psi^w f$ can be estimated using the assumption \eqref{eq almost diag smooth}, Theorem \ref{thm thm1} with $q=p'$, and \eqref{eq cont}. Hence, it is sufficient to prove that for every $N>0$,
\begin{equation}\label{eq daver 0}
    \|\cF (\varphi A (1-\psi)^w f)\|_{L^p}\leq C_N\|\langle x\rangle^{-N}f \|_{H^{-N}}
\end{equation}
for $f\in\cS(\rd)$. In fact, we will prove in a moment that
\begin{equation}\label{eq daver 1}
f\in\cS'(\rd)\ \Rightarrow\  \varphi A (1-\psi)^w f\in\cS(\rd).
\end{equation}
This shows that $\cF \varphi A (1-\psi)^w$, as a continuous linear operator $\cS'(\rd)\to \cS'(\rd)$, in fact maps $\cS'(\rd)$ into $ \cS(\rd)$. Consider, for every $N>0$, the Banach space $X_N$ of temperate distributions $f\in\cS'(\rd)$ such that $\|\langle x\rangle^{-N}f\|_{H^{-N}}<\infty$, with the obvious norm. By the closed graph theorem,  for every $N>0$, $\cF \varphi A (1-\psi)^w:X_N\to L^p(\rd)$ continuously, which gives \eqref{eq daver 0}.

It remains to prove \eqref{eq daver 1}. Let $V=\{0\}\times(\rd\setminus\{0\})$. By assumption we have 
\[
{\rm conesupp}(1-\psi)\cap \cA^{-1}(V)=\emptyset. 
\]
On the other hand, using \eqref{eq chi}, \eqref{eq micro} and \eqref{eq propag} (where $\varphi$ is regarded as a symbol in $S^0_{0,0}(\rdd)$ independent of the covariable $\xi$), we have 
\[
WF_G(\varphi A (1-\psi)^w f)\subset V\cap \cA({\rm conesupp}(1-\psi))=\emptyset.
\]
Hence, $\varphi A (1-\psi)^w f\in\cS(\R^d)$, which concludes the proof of \eqref{eq daver 1}. 
\end{proof}
\begin{corollary} Let $\nu$ be a positive Radon measure in $\rd$ with compact support, and let $1\leq p\leq 2$, $1\leq q\leq \infty$ such that $\cF:L^p(\R^d)\to L^q(\R^d,\nu)$ continuously. Under the same notation and assumptions as in Theorem \ref{thm micro}, there exists $C>0$ and, for every $N\in \bN$, there exists $C_N>0$ such that
\begin{equation}\label{eq 9ma due}
\|
A f\|_{L^q(\rd,\nu)}\leq C\|\psi^w f\|_{L^p(\rd)}+C_N\|\langle x\rangle^{-N} f\|_{H^{-N}(\rd)},\qquad f\in\cS(\R^d).
\end{equation}
\end{corollary}
\begin{proof}
    The desired result follows from Theorem \ref{thm micro} by writing $\varphi A= \cF^{-1} \cF \varphi A$ with $\varphi=1$ on the support of $\nu$. 
\end{proof}
 This result applies to the Schr\"odinger propagator $A=U_1(t)$ corresponding to a smooth Hamiltonian of the type $a_2(x,\xi)+a_1(x,\xi)$, with $a_2$ and $a_1$ satisfying assumptions (iii) and (iv) of Section \ref{sec nonsmooth}. Indeed, by \eqref{eq 5ma} we see that \eqref{eq almost diag smooth} holds with $\cA=S_t=
\begin{pmatrix} A_t& B_t\\
C_t& D_t
\end{pmatrix}\in Sp(d,\R)$ --- that is, the Hamiltonian flow associated with $a_2$ --- and the estimate \eqref{eq 9ma due} therefore holds for $A=U_1(t)$, with constants depending on $t$, for all $t$ such that ${\rm det} (B_t)\not=0$.

\section*{Acknowledgements} 
The author expresses his gratitude to Patrik Wahlberg for reading a preliminary version of the manuscript and providing valuable comments. 

The author is a Fellow of the \textit{Accademia delle Scienze di Torino} and a member of the \textit{Societ\`a Italiana di Scienze e Tecnologie Quantistiche (SISTEQ)}.  

\bibliographystyle{abbrv} 
\bibliography{biblio.bib}
\end{document}